\documentclass[11pt,leqno]{amsart}
\usepackage{verbatim,amssymb,amsfonts,latexsym,amsmath,amsthm,tikz,bbm,nicefrac,xspace,graphicx,mathtools,mathdots,enumerate}
\usetikzlibrary{arrows}
\DeclareMathAlphabet{\mathpzc}{OT1}{pzc}{m}{it}

\newtheorem{theorem}{Theorem}[section]
\newtheorem{lemma}[theorem]{Lemma}
\newtheorem{corollary}[theorem]{Corollary}
\newtheorem{proposition}[theorem]{Proposition}

\newtheorem*{maintheorem}{Main Theorem}

\newtheorem*{theorem*}{Theorem}
\newtheorem*{corollary*}{Corollary}
\newtheorem*{claim1}{Claim 1}
\newtheorem*{claim2}{Claim 2}

\newtheorem*{claim}{Claim}
\newtheorem*{sub-claim}{sub-claim}
\theoremstyle{definition}

\theoremstyle{remark}

\newcommand{\N}{\mathbb{N}}

\newcommand{\B}{\mathcal{B}}
\newcommand{\C}{\mathcal{C}}

\renewcommand{\H}{\mathcal{H}}

\renewcommand{\P}{\mathcal{P}}
\newcommand{\U}{\mathcal{U}}
\newcommand{\V}{\mathcal{V}}
\newcommand{\W}{\mathcal{W}}

\newcommand{\explicitSet}[1]{\left\lbrace #1 \right\rbrace}
\newcommand{\brackets}[1]{\left\langle #1 \right\rangle}
\newcommand{\set}[2]{\explicitSet{#1 \colon #2}}
\newcommand{\seq}[2]{\brackets{#1 \colon #2}}
\newcommand{\<}{\langle}
\renewcommand{\>}{\rangle}
\renewcommand{\a}{\alpha}
\renewcommand{\b}{\beta}

\newcommand{\e}{\varepsilon}

\renewcommand{\k}{\kappa}
\newcommand{\s}{\sigma}
\renewcommand{\t}{\tau}
\newcommand{\w}{\omega}
\newcommand{\0}{\emptyset}
\newcommand{\sub}{\subseteq}
\newcommand{\rest}{\!\restriction\!}

\newcommand{\homeo}{\approx}

\newcommand{\iso}{\cong}
\newcommand{\closure}[1]{\overline{#1}}

\newcommand{\quotient}{\twoheadrightarrow}
\newcommand{\st}{\mathsf{st}}
\newcommand{\clop}{\mathsf{clop}}

\newcommand{\setmins}{\hspace{-.4mm}\setminus\hspace{-.4mm}}
\newcommand{\PP}{\mathbb{P}}

\newcommand{\pwmf}{\mathcal{P}(\w)/\mathrm{fin}}
\newcommand{\pnmf}{\mathcal{P}(\N)/\mathrm{fin}}

\renewcommand{\AA}{\mathbb{A}}

\newcommand{\continuum}{\mathfrak{c}}
\newcommand{\pseudo}{\mathfrak{p}}

\renewcommand{\split}{\mathfrak s}

\newcommand{\ch}{\ensuremath{\mathsf{CH}}\xspace}
\newcommand{\zfc}{\ensuremath{\mathsf{ZFC}}\xspace}
\newcommand{\pfa}{\ensuremath{\mathsf{PFA}}\xspace}
\newcommand{\ocama}{\ensuremath{\mathsf{OCA+MA}}\xspace}
\newcommand{\ma}{\ensuremath{\mathsf{MA}}\xspace}

\newcommand{\masl}{\ensuremath{\mathsf{MA}_\k(\s\mathrm{-linked})}\xspace}

\renewcommand{\quotient}{\twoheadrightarrow}
\newcommand{\Iso}{\mathcal{I}\hspace{-.2mm}\mathpzc{so}}
\newcommand{\quot}{\mathcal{Q}\hspace{-.2mm}\mathpzc{uot}}

\newcommand{\auto}[1]{\mathcal{A}\hspace{-.2mm}\mathpzc{ut}\!\left(#1\right)}

\newcommand{\autstar}{\mathcal{H}\!\left( \w^* \right)}

\newcommand{\sinverse}{\s^{-1}}
\newcommand{\tr}[1]{[\hspace{-.5mm}[#1]\hspace{-.5mm}]}
\newcommand{\gen}[1]{\<\hspace{-1mm}\<\hspace{.5mm}#1\hspace{.5mm}\>\hspace{-1mm}\>}
\newcommand{\toh}{\xrightarrow{\ \!_{\,\!_{h}}\,}}
\newcommand{\tox}[1]{\xrightarrow{\ \!_{\,\!_{#1}}\,}}
\newcommand{\toG}{\xrightarrow{\ \!_{\,\!_{\mathsf{G}}}\,}}

\begin{document}

\title{The isomorphism class of the shift map}
\author{Will Brian}
\address {
W. R. Brian\\
Department of Mathematics and Statistics\\
University of North Carolina at Charlotte\\
9201 University City Blvd.\\
Charlotte, NC 28223-0001}
\email{wbrian.math@gmail.com}
\urladdr{wrbrian.wordpress.com}

\begin{abstract}
The \emph{shift map} is the self-homeomorphism of $\w^* = \b\w \setmins \w$ induced by the successor function $n \mapsto n+1$ on $\w$.
We prove that
the isomorphism classes of $\s$ and $\s^{-1}$ cannot be separated by a Borel set in $\autstar$, the space of all self-homeomorphisms of $\w^*$ equipped with the compact-open topology. 

Van Douwen proved it is consistent for $\s$ and $\s^{-1}$ not to be isomorphic. Whether it is also consistent for them to be isomorphic is an open problem. The theorem stated above can be thought of as a counterpoint to van Douwen's result: while $\s$ and $\s^{-1}$ may not be isomorphic, there is no simple topological property that distinguishes them. 

As a relatively straightforward consequence of the main theorem, we deduce that $\ocama$ implies the set of continuous images of $\s$ fails to be Borel in $\autstar$. (Here a ``continuous image'' of $\s$ is meant in the sense of topological dynamics: any $h \in \autstar$ such that $q \circ \s = h \circ q$ for some continuous surjection $q: \w^* \to \w^*$.) This contrasts starkly with a recent theorem of the author showing that under $\ch$, the continuous images of $\s$ form a closed subset of $\autstar$.
\end{abstract}

\maketitle


\section{Introduction}


The \emph{shift map} $\s: \w^* \to \w^*$ sends an ultrafilter $u \in \w^*$ to the unique ultrafilter generated by $\set{A+1}{A \in u}$. Equivalently, $\s$ is the restriction to $\w^*$ of the unique map $\b\w \to \b\w$ that continuously extends the successor function $n \mapsto n+1$ on $\w$.

Let $\autstar$ denote the group of all self-homeomorphisms of $\w^*$. Two maps $f,g \in \autstar$ are \emph{isomorphic} if they are conjugate in this group, i.e., if there is some third $h \in \autstar$ such that $h \circ f = g \circ h$. In this case we say that $h$ is an isomorphism from $f$ to $g$.

\vspace{.5mm}
\begin{center}
\begin{tikzpicture}[xscale=.85,yscale=.85]

\node at (5.1,0.05) {$\w^*$};
\node at (7.1,0.05) {$\w^*$};
\node at (7.1,2.05) {$\w^*$};
\node at (5.1,2.05) {$\w^*$};
\draw[->] (5.4,2) -- (6.7,2); \node at (6,2.25) {\small $f$};
\draw[->] (5.4,0) -- (6.7,0); \node at (6,-.22) {\small $g$};
\draw[<-] (7,.35) -- (7,1.68); \node at (7.2,1) {\small $h$};
\draw[<-] (5,.35) -- (5,1.68); \node at (4.8,1) {\small $h$};
\node at (10,1.25) {\large $f \iso g$};

\end{tikzpicture}
\end{center}
\vspace{-.5mm}

\noindent This is the standard notion of isomorphism in the category of topological dynamical systems. Isomorphic maps are topologically indistinguishable, in the sense that one cannot tell them apart without first assigning labels to the points of $\w^*$.

Beginning in the late 1980's, van Douwen considered the question of whether it is possible for $\s$ and $\s^{-1}$ to be isomorphic. He proved in \cite{vanDouwen} that it is consistent with \zfc for $\s$ and $\s^{-1}$ not to be isomorphic. Specifically, van Douwen showed that if $h \in \autstar$ is an isomorphism from $\s$ to $\s^{-1}$, then $h$ cannot be \emph{trivial}. A trivial map is one that does continuously extends to a function $\b\w \to \b\w$; or, equivalently, trivial maps are those that are ``induced'' by functions $\w \to \w$. Shelah proved in \cite{Shelah} (prior to the appearance of van Douwen's paper) that it is consistent for every member of $\autstar$ to be trivial. Shelah and Stepr\={a}ns showed later that this follows from \pfa \cite{Shelah&Steprans}, and ultimately Veli\v{c}kovi\'c showed that \ocama suffices to prove every $h \in \autstar$ is trivial, while $\ma+\neg$\ch does not suffice \cite{Velickovic}.

One may interpret van Douwen's theorem as more than just a consistency result. By showing that any isomorphism from $\s$ to $\s^{-1}$ must be non-trivial, he showed not only that such an isomorphism need not exist, but also that if such an isomorphism does exist, then it must be somewhat exotic. In other words, $\s$ and $\s^{-1}$ cannot be isomorphic in too simple a way. 

Here we prove a result in the opposite direction. Roughly, the theorem states that, although $\s$ and $\s^{-1}$ may fail to be isomorphic, their isomorphism classes are essentially indistinguishable in $\autstar$, or in other words, $\s$ and $\s^{-1}$ cannot fail to be isomorphic in too simple a way. More precisely:

\begin{maintheorem}
The isomorphism classes of $\s$ and $\s^{-1}$ cannot be separated by a Borel set in $\autstar$, the set of all self-homeomorphisms of $\w^*$ endowed with the compact-open topology.
\end{maintheorem}


In Section~\ref{sec:background}, we introduce some notation and prove a few lemmas concerning the compact-open topology on $\H(X)$ when $X$ is a Stone space. The main result of this section identifies a particularly nice basis for $\H(X)$: basic open sets are represented by finite digraphs showing the action of a map on a clopen partition of $X$. In Section~\ref{sec:dynamics} we observe some topological properties of $\s$ and $\s^{-1}$. Finally, Section~\ref{sec:main} contains a proof of the main theorem. Most of the work is accomplished via a lemma that is perhaps of some independent interest: it states that the isomorphism class of $\s$, considered as a subspace of $\autstar$, is a Baire space. We also derive a corollary: $\ocama$ implies the set of all quotients of $\s$ fails to be Borel in $\autstar$. (The notion of a quotient mapping is defined in Section~\ref{thm:main}.) This contrasts starkly with a recent theorem of the author showing that under $\ch$, the set of all quotients of $\s$ form a closed subset of $\autstar$.

\section{The space $\H(X)$ when $X$ is a Stone space}\label{sec:background}

In what follows, $X$ always denotes a compact Hausdorff space. A homeomorphism from $X$ to itself is called a \emph{topological dynamical system}, and the set of all such homeomorphisms is denoted $\H(X)$. 
The \emph{compact-open topology} on $\H(X)$ is the topology generated by the sets of the form
$$V(K,U) = \set{h \in \mathcal{H}(X)}{h[K] \sub U}$$
where $K \sub X$ is compact and $U \sub X$ is open.


Our first lemma shows that if $X$ is zero-dimensional, the compact-open topology on $\H(X)$ has a particularly nice subbasis. Note that a similar result was proved by Lupton and Pitz in \cite[Section 3]{Lupton&Pitz}, where they study the space $\C(X)$ of continuous self-maps of $X$, endowed with the compact-open topology, of which $\H(X)$ is a subspace. 

Let $\clop(X)$ denote the set of all clopen subsets of $X$. Given $A,B \in \clop(X)$, define 
$\tr{A,B} = \set{h \in \mathcal{H}(X)}{h[A] = B}.$

\begin{lemma}\label{lem:basis}
For any zero-dimensional compact Hausdorff space $X$, 
$$\set{\tr{A,B}}{A,B \in \clop(X)}$$ 
is a subbasis for the compact-open topology on $\H(X)$.
\end{lemma}

\begin{proof}
Let $\tau$ denote the topology on $\H(X)$ generated by 
$$\set{V(K,U)}{K \sub X \text{ is compact and } U \sub X \text{ is open}}$$
and let $\tau'$ denote the topology on $\H(X)$ generated by 
$$\set{\tr{A,B}}{A,B \in \clop(X)}.$$
To prove the lemma, we must show that $\tau = \tau'$.

Observe that if $A,B \in \clop(X)$ then
\begin{align*}
\tr{A,B} &= \set{h \in \H(X)}{h[A] \sub B \text{ and } h[X \setmins A] \sub X \setmins B} \\
&= \set{h \in \H(X)}{h[A] \sub B} \cap \set{h \in \H(X)}{h[X \setmins A] \sub X \setmins B}.
\end{align*}
As all the sets $A$, $B$, $X \setmins A$, and $X \setmins B$ are both compact and open, this shows that $\tr{A,B}$ is the intersection of two subbasic open sets from $\t$. Thus every subbasic open set from $\t'$ is in $\t$, and it follows that $\t' \sub \t$.

Similarly, to show that $\tau \sub \tau'$, it suffices to show that every set of the form $V(K,U)$ is in $\tau'$. This can be deduced directly from the results of Lupton and Pitz in \cite[Section 3]{Lupton&Pitz}. However, this part of their argument is short, and working in $\H(X)$ rather than $\C(X)$ makes it shorter still; so we include a (slighly simplified) version of their argument here.

Let $K \sub X$ be compact and $U \sub X$ open, and suppose $h \in V(K,U)$. Let
$$\U = \set{A \in \clop(X)}{A \cap K \neq \0 \text{ and } h[A] \sub U}.$$
Because $\clop(X)$ is a basis for $X$ and because $h[K] \sub U$, for every $x \in K$ there is some $A \in \clop(X)$ such that $h[A] \sub U$. Thus $\U$ is an open cover for $K$, and as $X$ is compact, $\U$ has a finite subcover $\{A_1,A_2,\dots,A_n\}$. Let $A = \bigcup_{i \leq n} A_i$. Then $A \in \clop(X)$, $A \supseteq K$, and $h[A] = \bigcup_{i \leq n}h[A_i] \sub U$. Furthermore, $h[A]$ is clopen in $X$ (because $h$ is a homeomorhpism).

Thus if $h \in V(K,U)$, then $h \in \tr{A,B}$ for some $A,B \in \clop(X)$ with $A \supseteq K$ and $B \sub U$. Conversely, it is clear that if $A \supseteq K$ and $B \sub U$ then $\tr{A,B} \sub V(K,U)$. Hence
$$V(K,U) = \textstyle \bigcup \set{\tr{A,B}}{A,B \in \clop(X),\, A \supseteq K, \text{ and } B \sub U}.$$
Hence $V(K,U) \in \tau'$, as desired.
\end{proof}

It follows from Lemma~\ref{lem:basis} that $\H(X)$ has a basis of the form 
$$\textstyle \set{\bigcap_{i < n}\tr{A_i,B_i}}{A_0,A_1,\dots,A_{n-1} \text{ are clopen subsets of }X}.$$
The following lemma shows that we can improve this slightly by requiring the $A_i$ to form a partition of $X$ into clopen sets.

\begin{lemma}\label{lem:basis2}
Let $X$ be a zero-dimensional compact Hausdorff space. Then
$$\B = \textstyle \set{\bigcap_{i < n}\tr{A_i,B_i} \vphantom{f^{f^f}}\ }{\set{A_i}{i < n} \text{ is a partition of }X \text{ into clopen sets}}$$
is a basis for the compact-open topology on $\H(X)$.
\end{lemma}
\begin{proof}
Every member of $\B$ is open in $\H(X)$ by the previous lemma. 

Now let $U = \bigcap_{i < n}\tr{A_i,B_i}$ be any basic open subset of $\H(X)$ (from the basis described in Lemma~\ref{lem:basis}), and let $h \in U$. Let $\set{A_i}{i < N}$ be the (finite) partition of $X$ generated by $\set{A_i}{i < n}$. Each member of this partition is clopen. Let $V = \bigcap_{i < N}\tr{A_i,h[A_i]}$. Then $V \in \B$ and $h \in V \sub U$. As $U$ and $h$ were arbtirary, this shows $\B$ is a basis for $\H(X)$. 
\end{proof}

If $\V$ is a collection of subsets of $X$ and $h \in \H(X)$, then we associate to $\V$ and $h$ a hitting relation: for $A,B \in \V$, write $A \toh B$ to mean that $h[A] \cap B \neq \0$. This defines a directed graph:
$$\mathsf{Hit}(\V,h) = \<\V,\toh\>.$$
This directed graph, or \emph{digraph}, with vertex set $\V$ and edge relation $\toh$, may have loops and may have two (oppositely directed) edges between some pairs of vertices. As in \cite{Bernardes&Darji} or \cite{Shimomura}, one may view this digraph as a means of capturing the combinatorial content of the action of $h$ on $\V$. One may also view such digraphs as providing an alternative description of the compact-open topology, as the following theorem shows.
If $\V$ is a collection of subsets of $X$ and $G$ is a digraph with vertex set $\V$, then define
$$\gen{\V,\mathsf{G}} = \set{h \in \H(X)}{\mathsf{Hit}(\V,h) = \mathsf{G}}.$$

\begin{theorem}\label{thm:basis}
Let $X$ be a zero-dimensional compact Hausdorff space. Then
\begin{align*}
\B = \left\{ \gen{\V,\mathsf{G}} \vphantom{f^{f^f}}\right. : \ \V \text{ is a partition of }X \text{ into clopen sets}& \\ 
\text{and }\mathsf{G}\text{ is a digraph with vertex set }\V & \left. \vphantom{f^{f^f}}\! \right\}
\end{align*}
is a basis for the compact-open topology on $\H(X)$.
\end{theorem}
\begin{proof}
Let $\V = \set{A_i}{i < n}$ be a partition of $X$ into clopen sets. (Note that every such partition is finite, because $X$ is compact.) Let $\mathsf{G}$ be a digraph with vertex set $\V$ and edge relation $\toG$. Then
\begin{align*}
\gen{\V,\mathsf{G}} = \textstyle \bigcup \left\{ \bigcap_{i < n}\tr{A_i,B_i} \vphantom{f^{f^f}}\right. : \ \text{for each }i < n, \, B_i \text{ is clopen and} & \\ 
B_i \cap A_j \neq \0 \text{ if and only if } A_i \toG A_j & \left. \vphantom{f^{f^f}}\! \right\}.
\end{align*}
By Lemma~\ref{lem:basis2}, $\gen{\V,\mathsf{G}}$ is open. Thus each member of $\B$ is open in $\H(X)$.

Let $U = \bigcap_{i < n}\tr{A_i,B_i}$ be a basic open subset of $\H(X)$ (from the basis described in Lemma~\ref{lem:basis2}), and let $h \in U$. Let $\V = \set{A_i}{i < n}$ and let $V = \gen{\V,\mathsf{Hit}(\V,h)}$. Then $V \in \B$ and $h \in V \sub U$. As $U$ and $h$ were arbtirary, this shows $\B$ is a basis for $\H(X)$. 
\end{proof}

Theorem~\ref{thm:basis} gives us a nice basis for $\H(X)$: a basic open set just specifies the hitting relation for a map on some clopen partition of $X$. In particular, the basic open neighborhoods of some point $h \in \H(X)$ are determined simply by the action of $h$ on some such partition. The next result states that using finer partitions, which give more information about the action of $h$, results in smaller neighborhoods of $h$.

\begin{proposition}\label{prop:refine}
Let $X$ be a zero-dimensional compact Hausdorff space. For each $h \in \H(X)$,
$$\mathcal N_h = \set{\gen{\V,\mathsf{Hit}(\V,h)}}{\V \text{ is a partition of }X\text{ into clopen sets}}$$
is a local basis for $h$ in $\H(X)$. If $\V$ and $\W$ are both partitions of $X$ into clopen sets, and $\W$ refines $\V$, then $\gen{\W,\mathsf{Hit}(\W,h)} \sub \gen{\V,\mathsf{Hit}(\V,h)}$.
\end{proposition}
\begin{proof}
If $h$ is in some basic open set $\gen{\V,\mathsf{G}}$, then $\mathsf{G} = \mathsf{Hit}(\V,h)$ by definition. Thus the first assertion follows from Theorem~\ref{thm:basis}. 

For the second assertion, suppose $\V$ and $\W$ are both partitions of $X$ into clopen sets, and $\W$ refines $\V$. 
Let $g \in \gen{\W,\mathsf{Hit}(\W,h)}$, which means that $\mathsf{Hit}(\W,g) = \mathsf{Hit}(\W,h)$. If $A,B \in \V$, then 
\begin{align*}
A \toh B \quad  \Leftrightarrow& \quad h[A] \cap B \neq \0 \\
\Leftrightarrow& \quad h[A'] \cap B \neq \0 \text{ for some }A' \in \W \ (\text{because } \W \text{ refines }\V) \\
\Leftrightarrow& \quad A' \toh B \text{ for some }A' \in \W \text{ with }A' \sub A \\
\Leftrightarrow& \quad A' \tox{g} B \text{ for some (the same) }A' \in \W \text{ with }A' \sub A \\
\Leftrightarrow& \quad A \tox{g} B. 
\end{align*}
Thus $\mathsf{Hit}(\V,g) = \mathsf{Hit}(\V,h)$, so $g \in \gen{\V,\mathsf{Hit}(\V,h)}$. As $g$ was arbitrary, $\gen{\W,\mathsf{Hit}(\W,g)} \sub \gen{\W,\mathsf{Hit}(\W,h)}$.
\end{proof}

Given a Boolean algebra $\AA$, the set of all ultrafilters on $\AA$, called the \emph{Stone space} of $\AA$ and denoted $\AA^\st$, carries a natural topology with a basis of clopen sets of the form
$\set{u \in \AA^\st}{a \in u}$,
where $a \in \AA$. With this topology, $\AA^\st$ is a zero-dimensional compact Hausdorff space, and $\clop(\AA^\st) \iso \AA$. 
Conversely, given a zero-dimensional compact Hausdorff space $X$, the set $\clop(X)$ forms a Boolean algebra such that $\clop(X)^\st \homeo X$. This correspondence, known as \emph{Stone duality}, reveals that the category of Boolean algebras and the category of zero-dimensional compact Hausdorff spaces (also known as \emph{Stone spaces}) are essentially interchangeable.

If $\AA$ is a Boolean algebra, then an automorphism mapping $\AA$ to itself is called an \emph{algebraic dynamical system}, and the set of all such automorphisms is denoted $\auto{\AA}$. 
The \emph{topology of pointwise convergence} on $\auto{\AA}$ is the topology generated by the sets of the form
$$\tr{a,b} = \set{\phi \in \auto{\AA}}{\phi(a) = b}$$
where $a,b \in \AA$. In other words, the basic open subsets of $\auto{\AA}$ are determined by specifying the action of a map at finitely many points of $\AA$.

Stone duality extends naturally from Boolean algebras and Stone spaces to their respective self-maps. If $X$ is a Stone space and $h \in \H(X)$, then $h$ permutes the clopen subsets of $X$ and thus defines an automorphism of $\clop(X)$; formally, $h^\st \in \auto{\clop(X)}$ is defined by setting
$$h^\st(A) = h[A]$$
for every clopen $A \sub X$.
Similarly, if $\AA$ is a Boolean algebra and $\phi \in \auto{\AA}$, then $\phi$ permutes the ultrafilters on $\AA$ and thus defines a self-homeomorphism of $\AA^\st$; formally, $\phi^\st \in \H(\AA^\st)$ is defined on each $u \in \AA^\st$ by taking
$$a \in \phi^\st(u) \ \Leftrightarrow \ \phi^{-1}(a) \in u$$
for all $a \in \AA$.  
One may check that $\phi \mapsto \phi^\st$ is a bijection $\auto{\AA} \to \H(\AA^\st)$, and $h \mapsto h^\st$ is a bijection $\H(X) \to \auto{\clop(X)}$. Furthermore, these bijections are inverse to one another: after identifying $\AA$ with $\clop(\AA^\st)$ in the natural way, $(\phi^\st)^\st = \phi$ for all $\phi \in \auto{\AA}$, and after identifying $\clop(X)^\st$ with $X$ in the natural way, $(h^\st)^\st = h$ for all $h \in \H(X)$.

\begin{proposition}
Suppose $X$ is a Stone space.
The compact-open topology on $\H(X)$ is dual to the topology of pointwise convergence on $\auto{\clop(X)}$, in the sense that the mapping $h \mapsto h^\st$ that sends $\H(X)$ to $\auto{\clop(X)}$ is a homeomorphism.
Similarly, if $\AA$ is a Boolean algebra, then the mapping $\a \mapsto \a^\st$ that sends $\auto{\AA}$ to $\H(\AA^\st)$ is a homeomorphism.
\end{proposition}
\begin{proof}
Each of these maps lift to a bijection from the subbasic open sets of the domain onto the subbasic open sets of the range.
\end{proof}
%

We end this section with a theorem due to Arens. Together with the other results in this section, it indicates that the compact-open topology is the ``right'' topology for $\H(X)$.

\begin{proposition}\label{prop:group} $(\mathrm{Arens}$, \cite[Theorem 3]{Arens}$)$
When endowed with the compact-open topology, $\H(X)$ is a topological group (with composition as the group operation). Furthermore, it is the coarsest topology on $\H(X)$ that makes it a topological group and has the property that the evaluation map $(h,x) \mapsto h(x)$ is a continuous function $\H(X) \times X \to X$.
\end{proposition}

\noindent When we cite this proposition in the following sections, we will really only need the fact that $\H(X)$ is a topological group. Note that this part of Arens's result is fairly easy for Stone spaces: for example, given Proposition 2.1 above, the inversion operation $h \mapsto h^{-1}$ is continuous because it is a bijection that permutes the subbasic open sets, sending $\tr{A,B}$ to $\tr{B,A}$ whenever $A,B \in \clop(X)$.

\section{The topological dynamics of the shift map}\label{sec:dynamics}

Let $X$ be a compact Hausdorff space and let $h \in \H(X)$ be a topological dynamical system. Given an open cover $\U$ of $X$, we say that a sequence of points $\seq{x_i}{i \leq n}$ is a $\U$\emph{-chain} if, for every $i < n$, there is some $U \in \U$ such that $h(x_i),x_{i+1} \in U$. Roughly, one may think of a $\U$-chain as a finite piece of an $h$-orbit, but computed with a small error at each step, where the allowed size of the error is determined by the fineness of $\U$. A dynamical system $h$ is \emph{chain transitive} if for any $a,b \in X$ and any open cover $\U$ of $X$, there is a $\U$-chain beginning at $a$ and ending at $b$.

\begin{lemma}\label{lem:ct0}$\ $
Let $X$ be a compact Hausdorff space. A topological dynamical system $h \in \H(X)$ is chain transitive if and only if $h(\closure{U}) \not\sub U$ for every open $U\neq \0,X$. If $X$ is also zero-dimensional, then $h$ is chain transitive if and only if $h[A] \not\sub A$ for every clopen $A \neq \0,X$. 
\end{lemma}
\begin{proof}
The first assertion is proved for metrizable dynamical systems in \cite[Theorem 4.12]{Akin}, using open covers that consist of $\e$-balls. The proof does not make any use of metrizability, and so a superficial modification of it yields a proof of the first assertion of the present lemma.

For the second assertion, suppose $X$ is zero-dimensional. It suffices to show that if $h(\closure{U}) \sub U$ for some open $U \neq \0,X$, then $h[A] \sub A$ for some clopen $A \neq \0,X$. Notice that the assertion ``$h(\closure{U}) \sub U$'' is equivalent to the assertion that $h$ is in the subbasic open subset $V(\closure{U},U)$ of $\H(X)$. In the proof of Lemma~\ref{lem:basis} above, we show (using the fact that $X$ is a Stone space) that $h \in V(\closure{U},U)$ implies there is some clopen $A \supseteq \closure{U}$ such that $h[A] \sub U$. This set $A$ is as required.
\end{proof}

\begin{corollary}\label{cor:shiftisct}
For any compact Hausdorff space $X$, the set of all chain transitive maps is closed in $\H(X)$.
\end{corollary}
\begin{proof}
By the previous lemma, a topological dynamical system $h \in \H(X)$ fails to be chain transitive if and only if
\begin{align*}
h \in \textstyle \bigcup \set{V(\closure{U},U)}{U \text{ is open and } U \neq \0,X}.
\end{align*}
Thus the set of non-chain-transitive maps is open in $\H(X)$.
\end{proof}

A \emph{path} in a directed graph $\mathsf{G}$ is a finite sequence $\seq{v_i}{i \leq n}$ of vertices of $\mathsf{G}$ such that for every $i < n$, there is an edge from $v_i$ to $v_{i+1}$. An \emph{infinite path} in $\mathsf{G}$ is an infinite sequence $\seq{v_i}{i < \w}$ of vertices of $\mathsf{G}$ such that for every $i < n$, there is an edge from $v_i$ to $v_{i+1}$. A directed graph is \emph{transitive} if for any two of its vertices $v$ and $w$, there is a path beginning at $v$ and ending at $w$.

For convenience, let us henceforth adopt the convention that an open cover of a space does not contain the empty set.

\begin{lemma}\label{lem:hit}
Let $X$ be any compact Hausdorff space. A topological dynamical system $h \in \H(X)$ is chain transitive if and only if $\mathsf{Hit}(\V,h)$ is transitive for every open cover $\V$ of $X$.
\end{lemma}
\begin{proof}
Suppose $h \in \H(X)$ is chain transitive. Let $\V$ be an open cover of $X$, and let $U,V \in \V$. Pick any $a \in U$ and $b \in V$, and let $\seq{x_i}{i \leq n}$ be a $\V$-chain with $x_0 = a$ and $x_n = b$. For each $0 < i < n$, fix some $U_i \in \V$ such that $x_i \in U_i$. Then
$$U \toh U_1 \toh U_2 \toh \dots \toh U_{n-1} \toh V$$
and as $U$ and $V$ were arbitrary, this shows $\mathsf{Hit}(\V,h)$ is transitive.

Now suppose $h \in \H(X)$ fails to be chain transitive. By Lemma~\ref{lem:ct0}, there is some open $U \neq \0,X$ such that $h(\closure{U}) \sub U$. Because $X$ is compact, $h(\closure{U})$ is closed; recalling that every compact Hausdorff space is normal, there is some open $V \sub X$ with $h(\closure{U}) \sub V \sub \closure{V} \sub U$. This implies that $\V = \{U,X \setminus \closure{V}\}$ is an open cover of $X$. Because $h(\closure{U}) \sub V$, we have $h(\closure{U}) \cap (X \setminus \closure{V}) = \0$. Using the surjectivity of $h$, and the fact that $U \neq \0,X$, one may check that $h(X \setmins \closure{V} \neq \0 \neq h(U) \cap U$. Thus the graph $\mathsf{Hit}(\V,h)$ has the following form:

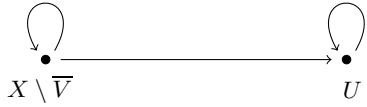
\vspace{-1.5mm}
\begin{center}
\begin{tikzpicture}[xscale=.8,yscale=.8]

\draw[fill=black] (0,0) circle (2pt);
\draw[fill=black] (5,0) circle (2pt);
\draw[->] (.25,0) -- (4.75,0);

\draw[->] (.15,.15) .. controls (.75,1.2) and (-.75,1.2) .. (-.15,.15);
\draw[->] (5.15,.15) .. controls (5.75,1.2) and (4.25,1.2) .. (4.85,.15);

\draw (5.1,-.5) node {\footnotesize$U$};
\draw (-.1,-.5) node {\footnotesize$X \setminus \closure{V}$};

\end{tikzpicture}
\end{center}
\vspace{-1.5mm}

\noindent In particular, this graph is not transitive.
\end{proof}

In what follows, our focus will be on the space $\w^*$, the Stone space of the Boolean algebra $\pwmf$. Given $A \sub \w$ and its mod-finite equivalence class $[A] \in \pwmf$, we will write $A^*$ instead of $[A]^*$ for the corresponding clopen subset of $\w^*$.

\begin{lemma}\label{lem:partition}
Suppose $A \sub \w$ is infinite. If $\V$ is a partition of $A^*$ into clopen sets, then there is a finite partition $\P$ of $A$ into infinite sets such that $\V = \set{B^*}{B \in \P}$.
\end{lemma}

A function $f: \w \to \w$ is called a \emph{mod-finite permutation} of $\w$ if there are some co-finite $A,B \sub \w$ such that $f$ restricts to a bijection $A \to B$. Every mod-finite permutation of $\w$ induces a homeomorphism $f^*: \w^* \to \w^*$, which sends every ultrafilter $u$ to the unique ultrafilter generated by $\set{f[A]}{A \in u}$:
$$A \in u \ \Leftrightarrow \ f[A] \in f^*(u).$$
Equivalently, $f^*$ is the restriction to $\w^*$ of the unique map $\b\w \to \b\w$ that continuously extends $f: \w \to \w$. Members of $\autstar$ that arise from mod-finite permutations of $\w$ in this way are called \emph{trivial}.

\begin{lemma}\label{lem:ct}$\ $
The shift map $\s$ and its inverse $\sinverse$ are chain transitive. Up to isomorphism, these are the only two trivial maps in $\autstar$ that are chain transitive.
\end{lemma}
\begin{proof}
Both assertions are proved in \cite[Section 5]{BrianPset}.
\end{proof}

We now turn more directly toward the main theme of the paper, which is to show that the shift map and its inverse are topologically indistinguishable within $\autstar$. Henceforth, let 
\begin{align*}
\Iso(\s) &= \set{h \in \autstar}{h \text{ is isomorphic to }\s}, \\
\Iso(\s^{-1}) &= \set{h \in \autstar}{h \text{ is isomorphic to }\s^{-1}}.
\end{align*}

\begin{theorem}\label{thm:equiv}
Let $\V$ be a partition of $\w^*$ into clopen sets, and let $\mathsf{G}$ be a directed graph with vertex set $\V$. The following are equivalent:
\begin{enumerate}
\item $\mathsf{G}$ is transitive.
\item $\mathsf{G} = \mathsf{Hit}(\V,h)$ for some $h \in \Iso(\s)$.
\item $\mathsf{G} = \mathsf{Hit}(\V,h)$ for some $h \in \Iso(\s^{-1})$.
\end{enumerate}
\end{theorem}


\begin{proof}
Both $(2)$ and $(3)$ imply $(1)$ by Lemmas \ref{lem:hit} and \ref{lem:ct}.

To show that $(1)$ implies $(2)$, let $\V$ be a partition of $\w^*$ into clopen sets. By Lemma~\ref{lem:partition}, we may write $\V = \set{A^*}{A \in \P}$, where $\P$ is a partition of $\w$ into finitely many infinite sets.

Because $\mathsf{G}$ is transitive, if $(A,B)$ and $(C,D)$ are edges in $\mathsf{G}$, then it is always possible to find a path from $B$ to $C$; i.e., a path connecting the end of the one edge to the beginning of the other. Using this fact, and the fact that the members of $\P$ are naturally identified with the vertices of $\mathsf{G}$, a simple recursive construction allows us to build an infinite sequence $\seq{V_n}{n < \w}$ of members of $\P$ such that $\seq{V_n^*}{n \in \w}$ is a path in $\mathsf{G}$ and 
\begin{itemize}
\item[$(\dagger)$] For every $A,B \in \P$ with $A^* \toG B^*$, there are infinitely many $n \in \w$ such that $V_n = A$ and $V_{n+1} = B$.
\end{itemize}
In other words, this path traverses every edge in $\mathsf{G}$ infinitely often. Observe that a transitive graph has no isolated vertices; thus in particular, $(\dagger)$ implies $\set{n \in \w}{V_n = A}$ is infinite for every $A \in \P$.

Next define a function $f: \w \to \w$ by setting
$$f(n) \,=\, \min (V_n \setmins \{f(0),f(1),\dots,f(n-1)\})$$
for all $n$. This function is well-defined because each $V_n$ is infinite. This function is clearly injective, and using the fact that $\set{n}{V_n = A}$ is infinite for every $A \in \P$, it is not hard to see that $f$ is also surjective. Thus $f$ is a bijection $\w \to \w$. 

Finally, let $h$ denote the function $f(n) \mapsto f(n+1)$.
Roughly, we may think of the bijection $f$ as a relabelling of the points of $\w$, and then think of $h$ as the (relabelled) successor map. 
We claim that $h^*$ is isomorphic to $\s$ and that $\mathsf{Hit}(\V,h^*) = \mathsf{G}$. 

First, note that $f$ and $h$ are both mod-finite permutations of $\w$, so that $f^*,h^* \in \autstar$. To see that $h^*$ is isomorphic to $\s$, let $s$ denote the successor function $n \mapsto n+1$ and note that $f \circ s = h \circ f$. This implies 
$f^* \circ s^* = h^* \circ f^*$.
But $s^* = \s$, so this shows that $f^*$ is an isomorphism from $\s$ to $h^*$.

For any two $A,B \in \P$ (not necessarily distinct), let 
$$E_{A \to B} = \set{n \in \w}{V_n = A \text{ and } V_{n+1} = B}.$$
If $A^* \toG B^*$, then $E_{A \to B}$ is infinite by property $(\dagger)$. On the other hand, if $\neg(A^* \toG B^*)$ then $E_{A \to B} = \0$, because the sequence $\seq{V_n}{n \in \w}$ has $V_n^* \toG V_{n+1}^*$ for every $n \in \w$. Thus
$$E_{A \to B} \text{ is infinite} \ \Leftrightarrow \ A^* \toG B^*.$$
By our definition of $f$, 
$E_{A \to B} = \set{f(m)}{f(m) \in A \text{ and } f(m+1) \in B}$.
Now observe that, by our definition of $h$, 
$$n \in A \text{ and } h(n) \in B \ \ \Leftrightarrow \ \ n \in E_{A \to B},$$
which implies $h[A] \cap B$ is infinite if and only if $E_{A \to B}$ is infinite. So
\begin{align*}
h^*[A^*] \cap B^* = (h[A] \cap B)^* \neq \0 \ &\Leftrightarrow \ h[A] \cap B \text{ is infinite} \\ 
&\Leftrightarrow \ E_{A \to B} \text{ is infinite} \ \Leftrightarrow \ A^* \toG B^*.
\end{align*}
As $A$ and $B$ were arbitrary members of $\P$ and $\V = \set{A^*}{A \in \P}$, this shows that $\mathsf{G} = \mathsf{Hit}(\V,h^*)$, which completes the proof that $(1)$ implies $(2)$.

One can prove that $(1)$ implies $(3)$ by a similar argument. Alternatively, one may deduce that $(1)$ implies $(3)$ from the now-proved fact that $(1)$ implies $(2)$ for all $\mathsf{G}$. To see this, fix some partition $\V$ of $\w^*$ into clopen sets, and let $\mathsf{G}$ be a transitive digraph with vertex set $\V$. Let $\mathsf{G}^{-1}$ denote the digraph obtained from $\mathsf{G}$ by inverting the edge relation. It is easy to check a digraph is transitive if and only if its inverse is; hence $\mathsf{G}^{-1}$ is transitive. Because $(1)$ implies $(2)$ for every digraph, this means there is some $h \in \Iso(\s)$ with $\mathsf{Hit}(\V,h) = \mathsf{G}^{-1}$. But then $h^{-1} \in \Iso(\s^{-1})$ and $\mathsf{Hit}(\V,h^{-1}) = \mathsf{G}$.
\end{proof}

\begin{corollary}\label{cor:dense}
If $U \sub \autstar$ is open, then $U \cap \Iso(\s) \neq \0$ if and only if $U \cap \Iso(\s^{-1}) \neq \0$. 
\end{corollary}
\begin{proof}
Let $U \sub \autstar$ be open. Suppose $U \cap \Iso(\s) \neq \0$, and fix some ${h \in U \cap \Iso(\s)}$. By Theorem~\ref{thm:basis}, there is a partition $\V$ of $\w^*$ into clopen sets, and a digraph $\mathsf{G}$ with vertex set $\V$, such that $h \in \gen{\V,\mathsf{G}} \sub U$ or, equivalently, $\mathsf{Hit}(\V,h) = \mathsf{G}$. By the previous theorem, there is some ${g \in \Iso(\s^{-1})}$ such that $\mathsf{Hit}(\V,g) = \mathsf{G}$. But this means $g \in \gen{\V,\mathsf{G}} \sub U$, so $g$ witnesses the fact that $U \cap \Iso(\s^{-1}) \neq \0$. An essentially identical argument shows that if $U \cap \Iso(\s^{-1}) \neq \0$ then $U \cap \Iso(\s) \neq \0$.
\end{proof}

Notice that this corollary implies a weak version of the main theorem: $\Iso(\s)$ and $\Iso(\s^{-1})$ cannot be separated by an open subset of $\autstar$.

We end this section with a strengening of Corollary~\ref{cor:dense}. This extension is not necessary for understanding the proof of the main theorem in the next section (and thus may be skipped if desired). It is, however, further support for the informal idea behind the main theorem: that no simple topological property can distinguish $\s$ from $\s^{-1}$.

Recall that $\split$ denotes the \emph{splitting number}, the smallest cardinality of a family $\mathcal S$ of subsets of $\w$ such that, for every $A \sub \w$, there is some $D \in \mathcal S$ such that both $A \cap D$ and $A \setmins D$ are infinite. 

\begin{theorem}\label{thm:s}
Suppose $G$ is an intersection of $<\!\split$ open sets in $\autstar$. Then $G \cap \Iso(\s) \neq \0$ if and only if $G \cap \Iso(\sinverse) \neq \0$. 
\end{theorem}

\begin{proof}
Suppose $G$ is an intersection of $\k < \split$ open subsets of $\autstar$, and suppose $\s \in G$. (We consider later the possibility that $G \cap \Iso(\s) \neq \0$ but $\s \notin G$.) By shrinking each open set to a finite intersection of subbasic open neighborhoods of $\s$, we may (and do) assume that $G$ is an intersection of $\k$ subbasic open sets: that is, $\s \in G = \bigcap_{\a \in \k}\tr{A_\a^*,B_\a^*}$.

Because $\k < \split$, there is some $D \sub \w$ such that for every $\a \in \k$, 
either $D \cap A_\a$ or $D \setminus A_\a$ is finite.

We now define a permutation $h: \w \to \w$ by ``flipping'' some intervals associated to $D$. Specifically, let $d_0=0$ and let $\{d_1,d_2,d_3,\dots\}$ be an increasing enumeration of the set $D+1 = \set{i+1}{i \in D}$. Define $h: \w \to \w$ by setting
$$h(n) \,=\,
(d_{k+1}-1)-i \ \ \ \text{ whenever } \ \ \ n = d_k+i.
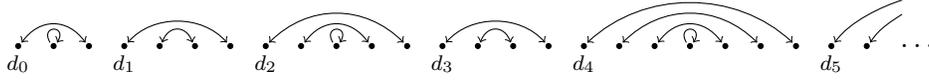
$$

\vspace{1mm}
\begin{center}
\begin{tikzpicture}[xscale=.47,yscale=.47]

\draw[fill=black] (-5,0) circle (2pt);
\draw[fill=black] (-4,0) circle (2pt);
\draw[fill=black] (-3,0) circle (2pt);
\draw[fill=black] (-2,0) circle (2pt);
\draw[fill=black] (-1,0) circle (2pt);
\draw[fill=black] (0,0) circle (2pt);
\draw[fill=black] (1,0) circle (2pt);
\draw[fill=black] (2,0) circle (2pt);
\draw[fill=black] (3,0) circle (2pt);
\draw[fill=black] (4,0) circle (2pt);
\draw[fill=black] (5,0) circle (2pt);
\draw[fill=black] (6,0) circle (2pt);
\draw[fill=black] (7,0) circle (2pt);
\draw[fill=black] (8,0) circle (2pt);
\draw[fill=black] (9,0) circle (2pt);
\draw[fill=black] (10,0) circle (2pt);
\draw[fill=black] (11,0) circle (2pt);
\draw[fill=black] (12,0) circle (2pt);
\draw[fill=black] (13,0) circle (2pt);
\draw[fill=black] (14,0) circle (2pt);
\draw[fill=black] (15,0) circle (2pt);
\draw[fill=black] (16,0) circle (2pt);
\draw[fill=black] (17,0) circle (2pt);
\draw[fill=black] (18,0) circle (2pt);
\draw[fill=black] (19,0) circle (2pt);

\draw (20.5,0) node {$\dots$};

\draw (-5,-.5) node {\scriptsize$d_0$};
\draw (-2,-.5) node {\scriptsize$d_1$};
\draw (2,-.5) node {\scriptsize$d_2$};
\draw (7,-.5) node {\scriptsize$d_3$};
\draw (11,-.5) node {\scriptsize$d_4$};
\draw (18,-.5) node {\scriptsize$d_5$};

\draw[<->] (-4.9,.1) .. controls (-4.4,.8) and (-3.6,.8) .. (-3.1,.1);
\draw[->] (-4.1,.1) .. controls (-4.4,.6) and (-3.6,.6) .. (-3.9,.1);

\draw[<->] (-1.9,.1) .. controls (-1.2,.9) and (.2,.9) .. (.9,.1);
\draw[<->] (-.93,.1) .. controls (-.75,.6) and (-.25,.6) .. (-.07,.1);

\draw[<->] (2.1,.1) .. controls (3.2,1.2) and (4.8,1.2) .. (5.9,.1);
\draw[<->] (3.1,.1) .. controls (3.6,.8) and (4.4,.8) .. (4.9,.1);
\draw[->] (3.9,.1) .. controls (3.6,.6) and (4.4,.6) .. (4.1,.1);

\draw[<->] (7.1,.1) .. controls (7.8,.9) and (9.2,.9) .. (9.9,.1);
\draw[<->] (8.07,.1) .. controls (8.25,.6) and (8.75,.6) .. (8.93,.1);

\draw[<->] (11.1,.1) .. controls (12.7,1.6) and (15.3,1.6) .. (16.9,.1);
\draw[<->] (12.1,.1) .. controls (13.2,1.2) and (14.8,1.2) .. (15.9,.1);
\draw[<->] (13.1,.1) .. controls (13.6,.8) and (14.4,.8) .. (14.9,.1);
\draw[->] (13.9,.1) .. controls (13.6,.6) and (14.4,.6) .. (14.1,.1);

\draw[->] (20,1.3) .. controls (19.5,1.15) and (18.5,.7) .. (18.1,.1);
\draw[->] (20,.9) .. controls (19.7,.7) and (19.35,.5) .. (19.1,.1);

\end{tikzpicture}
\end{center}
\vspace{1mm}

Now consider the map $f = h \circ s^{-1} \circ h^{-1}$, where $s$ denotes the successor function $s(n) = n+1$. This function is not defined at $d_1-1$ (because $h(d_1-1) = 0$), but is defined on the rest of $\w$. We have
$$f(n) \,=\,
\begin{cases}
n+1 \ & \text{ if } n \notin D,\\
d_{i-2} & \text{ if } n = d_i-1 \text{ for some } i \geq 2.
\end{cases}
$$

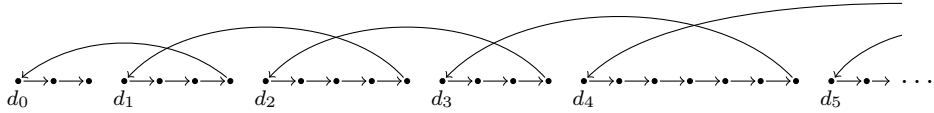
\vspace{1mm}
\begin{center}
\begin{tikzpicture}[xscale=.47,yscale=.47]

\draw[fill=black] (-5,0) circle (2pt);
\draw[fill=black] (-4,0) circle (2pt);
\draw[fill=black] (-3,0) circle (2pt);
\draw[fill=black] (-2,0) circle (2pt);
\draw[fill=black] (-1,0) circle (2pt);
\draw[fill=black] (0,0) circle (2pt);
\draw[fill=black] (1,0) circle (2pt);
\draw[fill=black] (2,0) circle (2pt);
\draw[fill=black] (3,0) circle (2pt);
\draw[fill=black] (4,0) circle (2pt);
\draw[fill=black] (5,0) circle (2pt);
\draw[fill=black] (6,0) circle (2pt);
\draw[fill=black] (7,0) circle (2pt);
\draw[fill=black] (8,0) circle (2pt);
\draw[fill=black] (9,0) circle (2pt);
\draw[fill=black] (10,0) circle (2pt);
\draw[fill=black] (11,0) circle (2pt);
\draw[fill=black] (12,0) circle (2pt);
\draw[fill=black] (13,0) circle (2pt);
\draw[fill=black] (14,0) circle (2pt);
\draw[fill=black] (15,0) circle (2pt);
\draw[fill=black] (16,0) circle (2pt);
\draw[fill=black] (17,0) circle (2pt);
\draw[fill=black] (18,0) circle (2pt);
\draw[fill=black] (19,0) circle (2pt);

\draw (20.5,0) node {$\dots$};

\draw (-5,-.5) node {\scriptsize$d_0$};
\draw (-2,-.5) node {\scriptsize$d_1$};
\draw (2,-.5) node {\scriptsize$d_2$};
\draw (7,-.5) node {\scriptsize$d_3$};
\draw (11,-.5) node {\scriptsize$d_4$};
\draw (18,-.5) node {\scriptsize$d_5$};

\draw[<-] (-4.15,0)--(-4.85,0);
\draw[<-] (-3.15,0)--(-3.85,0);
\draw[<-] (-1.15,0)--(-1.85,0);
\draw[<-] (-.15,0)--(-.85,0);
\draw[<-] (.85,0)--(.15,0);
\draw[<-] (2.85,0)--(2.15,0);
\draw[<-] (3.85,0)--(3.15,0);
\draw[<-] (4.85,0)--(4.15,0);
\draw[<-] (5.85,0)--(5.15,0);
\draw[<-] (7.85,0)--(7.15,0);
\draw[<-] (8.85,0)--(8.15,0);
\draw[<-] (9.85,0)--(9.15,0);
\draw[<-] (11.85,0)--(11.15,0);
\draw[<-] (12.85,0)--(12.15,0);
\draw[<-] (13.85,0)--(13.15,0);
\draw[<-] (14.85,0)--(14.15,0);
\draw[<-] (15.85,0)--(15.15,0);
\draw[<-] (16.85,0)--(16.15,0);
\draw[<-] (18.85,0)--(18.15,0);
\draw[<-] (19.65,0)--(19.15,0);

\draw[->] (20,2.2) .. controls (18,2.2) and (13,2) .. (11.1,.1);
\draw[->] (20,1.3) .. controls (19.4,1.15) and (18.5,.7) .. (18.1,.1);
\draw[<-] (-4.9,.1) .. controls (-3.3,1.4) and (-.7,1.4) .. (.9,.1);
\draw[<-] (-1.9,.1) .. controls (-.2,2) and (4.2,2) .. (5.9,.1);
\draw[<-] (2.1,.1) .. controls (3.8,2) and (8.2,2) .. (9.9,.1);
\draw[<-] (7.1,.1) .. controls (9.3,2.4) and (14.7,2.4) .. (16.9,.1);

\end{tikzpicture}
\end{center}
\vspace{1mm}

Observe that $h$ and $f$ are mod-finite permutations of $\w$, so $h^*,f^* \in \autstar$. Furthermore, $h^*$ is an isomorphism from $\s^{-1}$ to $f^*$ because $f = h \circ s^{-1} \circ h^{-1}$ implies $f^* = (h \circ s^{-1} \circ h^{-1})^* = h^* \circ \s^{-1} \circ (h^*)^{-1}$. A little less precisely (but perhaps more clearly), $f^*$ is isomorphic to $\s^{-1}$ because one can transform $f$ into $s^{-1}$ by relabelling the points of $\w$ (which is clear from the picture).

Thus $f^* \in \Iso(\s^{-1})$, and we claim that $f^* \in G$. Let $\a \in \k$. By our choice of $D$, either $D \cap A_\a$ or $D \setminus A_\a$ is finite. If $D \cap A_\a$ is finite, then $f(m) = m+1$ for all but finitely many $m \in A_\a$, which means 
$$f^*[A_\a^*] = (A_\a+1)^* = \s[A_\a^*] = B_\a^*.$$
If on the other hand $D \setmins A_\a$ is finite, then for all $m \in A_\a$ (except possibly $m = d_1-1$), either $m \notin D$, in which case $f(m) = m+1$, or else $m \in D$, and then $m = d_i-1$ for some $i \geq 2$ and $f(m) = d_{i-2} \in D+1$. 
Thus 
\begin{align*}
f^*[A_\a^*] &= f^*[(A_\a \setmins D)^* \cup D^*] = f^*[(A_\a \setmins D)^*] \cup f[D^*] \\
&= ((A_\a \setmins D)+1)^* \cup (D+1)^* = (A_\a+1)^* = \s[A_\a^*] = B_\a^*.
\end{align*}
Thus in either case, $f^*(A_\a^*) = \s(A_\a^*) = B_\a^*$, which implies $f^* \in \tr{A_\a^*,B_\a^*}$. As this holds for all $\a \in \k$, we have $f^* \in G$ as claimed.

The preceding argument shows that if $\s \in G$, where $G$ is some $<\!\split$-sized intersection of open subsets of $\autstar$, then $G \cap \Iso(\s^{-1}) \neq \0$. Next suppose $g \in \Iso(\s) \cap G$, but (possibly) $g \neq \s$. Fix an isomorphism between $\s$ and $g$; that is, fix some $h \in \autstar$ such that $h \circ g = \s \circ h$, or equivalently, $g = h^{-1} \circ \s \circ h$. Let 
$$hGh^{-1} = \set{e \in \autstar}{h^{-1} \circ e \circ h \in G}$$
and observe that $\s \in hGh^{-1}$ because $g \in G$. 

By Proposition~\ref{prop:group}, the map $e \mapsto h^{-1} \circ e \circ h$ is a self-homeomorphism of $\autstar$. In particular, $hGh^{-1} = \bigcap_{\a < \k}\tr{hA_\a h^{-1},hB_\a h^{-1}}$ is a countable intersection of open sets in $\autstar$. As $\s \in hGh^{-1}$, the argument above allows us to conclude that there is some $f \in \Iso(\s^{-1}) \cap hGh^{-1}$. By definition, $f \in hGh^{-1}$ implies $h^{-1} \circ f \circ h \in G$. But $h^{-1} \circ f \circ h$ is isomorphic to $f$ (indeed, $h$ is a witnessing isomorphism), and $f$ is isomorphic to $\s^{-1}$; thus $h^{-1} \circ f \circ h$ is isomorphic to $\s^{-1}$. Hence $h^{-1} \circ f \circ h \in \Iso(\s^{-1}) \cap G \neq \0$.

This shows that if $G \cap \Iso(\s) \neq \0$ then $G \cap \Iso(\s^{-1}) \neq \0$, thus proving the ``only if'' direction of the lemma. To prove the ``if'' direction, one may either $(1)$ note that the argument above is easily modified, by swapping the roles of $\s$ and $\sinverse$, to obtain a proof of this direction, or $(2)$ note that one may deduce the ``if'' direction from the ``only if'' direction directly, using the fact that the inversion map is a self-homeomorphism of $\autstar$ by Proposition~\ref{prop:group}.
\end{proof}

\section{A proof of the main theorem}\label{sec:main}


\begin{lemma}\label{lem:baire}
$\Iso(\s)$ is a Baire space. That is, if $\set{U_n}{n \in \w}$ is a collection of dense open subsets of $\Iso(\s)$, then $\bigcap_{n < \w}U_n$ is dense in $\Iso(\s)$. 
\end{lemma}
\begin{proof}
Let $\set{U_n}{n \in \w}$ be a collection of dense open subsets of $\Iso(\s)$, and let $V$ be a nonempty open subset of $\Iso(\s)$. To prove the lemma, we must show $V \cap \bigcap_{n \in \w}U_n \neq \0$.

By Theorem~\ref{thm:basis}, the sets of the form $\gen{\V,\mathsf{G}}$ constitute a basis for $\autstar$. It follows that the sets of the form $\gen{\V,\mathsf{G}} \cap \Iso(\s)$ constitute a basis for $\Iso(\s)$. For convenience, we will denote $\gen{\V,\mathsf{G}} \cap \Iso(\s)$ by $\gen{\V,\mathsf{G}}_\s$ for the remainder of the proof.


To begin, we define by recursion a sequence $\seq{\gen{\V_n,\mathsf{G}_n}_\s}{n \in \w}$ of basic open subsets of $\Iso(\s)$, along with a sequence $\seq{\P_n}{n \in \w}$ of partitions of $\w$, satisfying the following five properties for all $n \in \w$: 

\begin{minipage}{\textwidth}
\begin{enumerate}
\item $\P_n$ partitions $\w$ into finitely many infinite sets.
\item $\V_n = \set{A^*}{A \in \P_n}$.
\item $\0 \neq \gen{\V_n,\mathsf{G}_n}_\s \sub V \cap U_0 \cap \dots \cap U_n$.
\item if $m < n$, then $\gen{\V_n,\mathsf{G}_n}_\s \sub \gen{\V_m,\mathsf{G}_m}_\s$.
\item if $m < n$, then $\P _n$ is a refinement of $\P_m$.
\end{enumerate}
\end{minipage}

To begin the recursion, note that because $U_0 \cap V$ is a nonempty open subset of $\Iso(\s)$, Theorem~\ref{thm:basis} implies there is some partition $\V_0$ of $\w^*$ into clopen sets, and some directed graph $\mathsf{G}_0$ with vertex set $\V_0$, such that 
$$\gen{\V_0,\mathsf{G}_0}_\s \neq \0 \ \ \text{ and } \ \ \gen{\V_0,\mathsf{G}_0}_\s \sub V \cap U_0.$$
By Lemma~\ref{lem:partition}, we may write $\V_0 = \set{A^*}{A \in \P_0}$, where $\P_0$ is some partition of $\w$ into finitely many infinite sets. This concludes the base step of the recursion: properties $(1)$ through $(3)$ are satisfied for $n=0$, and properties $(4)$ and $(5)$ are vacuous for $n=0$.

After stage $\ell-1$ of the recursion, we have a sequence $\seq{\gen{\V_m,\mathsf{G}_m}_\s}{m < \ell}$ of basic open subsets of $\Iso(\s)$, and a sequence $\seq{\P_m}{m < \ell}$ of partitions of $\w$, such that properties $(1)$ through $(5)$ above are satisfied for ${n = \ell-1}$. 
We must construct $\V_\ell$, $\mathsf{G}_\ell$, and $\P_\ell$. Note that $\gen{\V_{\ell-1},\mathsf{G}_{\ell-1}}_\s$ is a nonempty open subset of $\Iso(\s)$. Because $U_\ell$ is a dense open subset of $\Iso(\s)$, this implies $\gen{\V_{\ell-1},\mathsf{G}_{\ell-1}}_\s \cap U_\ell$ is a nonempty open subset of $\Iso(\s)$; hence it contains a nonempty basic open set. That is, 
we may choose some partition $\V_\ell^0$ of $\w^*$ into clopen sets, and some digraph $\mathsf{G}_\ell^0$ with vertex set $\V^0_\ell$, such that
$$\0 \,\neq\, \gen{\V_\ell^0,\mathsf{G}^0_\ell}_\s \,\sub\, \gen{\V_{\ell-1},\mathsf{G}_{\ell-1}}_\s \cap U_\ell.$$

For each $A \in \P_{\ell-1}$, let
$\W_A = \set{B \cap A^*}{B \in \V_\ell^0 \text { and } B \cap A^* \neq \0}$
and observe that $\W_A$ is a partition of $A^*$ into (nonempty) clopen sets. By Lemma~\ref{lem:partition}, there is a partition $\C_A$ of $A$ into finitely many infinite sets such that $\W_A = \set{C^*}{C \in \C_A}$. 
Let 
$$\textstyle \P_\ell = \bigcup_{A \in \P_{\ell-1}}\C_A \ \ \text{ and } \ \ \V_\ell = \set{A^*}{A \in \P_\ell}.$$

This definition of $\V_\ell$ makes it obvious that hypothesis $(2)$ is satisfied for $n=\ell$. Also, because $\P_\ell$ was obtained by taking each of the finitely many infinite sets in $\P_{\ell-1}$ and further partitioning each of them into finitely many infinite sets, it is clear that hypothesis $(1)$ is satisfied for $n=\ell$, and that hypothesis $(5)$ is satisfied for $n = \ell$ (using the fact that hypothesis $(5)$ is satisfied for $n = \ell-1$).

To define the digraph $\mathsf{G}_\ell$, first recall $\gen{\V_\ell^0,\mathsf{G}^0_\ell}_\s \neq \0$. Fix some (any) $h \in \gen{\V_\ell^0,\mathsf{G}^0_\ell}_\s$ and let $\mathsf{G}_\ell = \mathsf{Hit}(\V_\ell,h)$. Note that $h \in \gen{\V_\ell,\mathsf{G}_\ell}_\s$, so in particular $\gen{\V_\ell,\mathsf{G}_\ell}_\s \neq \0$. Furthermore, because $\V_\ell$ refines $\V_\ell^0$, 
\begin{align*}
\gen{\V_\ell,\mathsf{G}_\ell}_\s &= \gen{\V_\ell,\mathsf{Hit}(\V_\ell,h)}_\s \sub \gen{\V_\ell^0,\mathsf{Hit}(\V_\ell^0,h)}_\s = \gen{\V_\ell^0,\mathsf{G}^0_\ell}_\s
\end{align*} 
by Proposition~\ref{prop:refine}. Thus
$$\gen{\V_\ell,\mathsf{G}_\ell}_\s \sub \gen{\V_\ell^0,\mathsf{G}^0_\ell}_\s \sub \gen{\V_{\ell-1},\mathsf{G}_{\ell-1}} \cap U_\ell$$
Using this containment, it is clear that the inductive hypotheses $(2)$ and $(3)$ for $n=\ell$ follow from the inductive hypotheses $(2)$ and $(3)$ for $n=\ell-1$.

This completes the recursive construction:
thus we obtain a sequence $\seq{\gen{\V_n,\mathsf{G}_n}_\s}{n \in \w}$ of basic open subsets of $\Iso(\s)$, together with a sequence $\seq{\P_n}{n \in \w}$ of partitions of $\w$ such that properties $(1)$ through $(5)$ listed above are satisfied for every $n \in \w$.

In the next part of the proof, we construct a function $f: \w \to \w$. The idea is that ultimately we will let $h$ denote the map $f(n) \mapsto f(n+1)$ as in the proof of Theorem~\ref{thm:equiv}, and we will then complete the proof of the lemma by showing $h^* \in \bigcap_{n \in \w}\gen{\V_n,\mathsf{G}_n}_\s \sub V \cap \bigcap_{n \in \w}U_n$. The construction of the map $f$ here is similar to the construction of the map $f$ in the proof of Theorem~\ref{thm:equiv}, but with one extra complication: instead of representing a path through a single digraph $\mathsf{G}$ (like the function in the proof of Theorem~\ref{thm:equiv}), the function $f$ constructed here ``diagonalizes'' across an infinite sequence of digraphs $\seq{\mathsf{G}_n}{n \in \w}$.

Before defining $f$, we first define an infinite sequence of paths, one through each of the graphs $\mathsf{G}_n$. For each $n$, let $\tox{\mathsf{G}_n}$ denote the edge relation for $\mathsf{G}_n$.

To begin, consider the digraph $\mathsf{G}_0$, which is transitive by Theorem~\ref{thm:equiv}. 
Because $\mathsf{G}_0$ is transitive, if $(A,B)$ and $(C,D)$ are edges in $\mathsf{G}_0$, then it is always possible to find a path from $B$ to $C$; i.e., a path connecting the end of the one edge to the beginning of the other. Using this fact, and the fact that the members of $\P_0$ are naturally identified with the vertices of $\mathsf{G}_0$, a simple recursive construction allows us to build a finite sequence $\seq{V_i^0}{i \leq K_0}$ of members of $\P_0$ such that 
\begin{itemize}
\item[$(\mathbf A_0)$] $\seq{(V_i^0)^*}{i \leq K_0}$ is a path in $\mathsf{G}_0$. 
\item[$(\mathbf B_0)$] For every $A,B \in \P_0$ with $A^* \tox{\mathsf{G}_0} B^*$, there is at least one $i < K_0$ such that $V_i^0 = A$ and $V_{i+1}^0 = B$.
\end{itemize}
In other words, this path traverses every edge in $\mathsf{G}_0$ at least once. Observe that a transitive graph has no isolated vertices; thus in particular, $(\mathbf B_0)$ implies $\set{i < K_0}{V_i^0 = A} \neq \0$ for every $A \in \P_0$.

After stage $n-1$ of the recursion, suppose we have constructed for every $m < n$ some finite sequence $\seq{V_i^m}{i \leq K_m}$ of members of $\P_m$ such that
\begin{itemize}
\item[$(\mathbf A_m)$] $\seq{(V_i^m)^*}{i \leq K_m}$ is a path in $\mathsf{G}_m$.
\item[$(\mathbf B_m)$] For every $A,B \in \P_m$ with $A^* \tox{\mathsf{G}_m} B^*$, there is at least one $i < K_m$ such that $V_i^m = A$ and $V_{i+1}^m = B$.
\item[$(\mathbf C_m)$] If $A$ denotes the (unique) member of $\P_{m-1}$ containing $V^m_0$, then ${(V^{m-1}_{K_{m-1}})^* \tox{\mathsf{G}_{m-1}} A^*}$.
\end{itemize}

To construct $\seq{V_i^n}{i \leq K_n}$, first consider the last set $V^{n-1}_{K_{n-1}}$ in the finite sequence constructed at the previous step of our recursion. Choose any $A \in \P_{n-1}$ such that $(V^{n-1}_{K_{n-1}})^* \tox{\mathsf{G}_{n-1}} A^*$, and then choose $V^n_0$ to be any member of $\P_n$ contained in $A$. Note that some such choice of $A$ is possible because $\mathsf{G}_{n-1}$ is transitive, and some such choice of $V^n_0$ is possible because $\P_n$ refines $\P_{n-1}$. This choice of $V^n_0$ ensures that hypothesis $(\mathbf C_n)$ holds. Now that $V^n_0$ is chosen, we proceed as in the base step to find the rest of $\seq{V^n_i}{i \leq K_n}$ satisfying $(\mathbf A_n)$ and $(\mathbf B_n)$. As in the base step, it is the transitivity of $\mathsf{G}_n$ that enables us to find a finite sequence $\seq{V^n_i}{i \leq K_n}$ of sets in $\P_n$ satisfying $(\mathbf A_n)$ and $(\mathbf B_n)$.

This completes the recursion. Thus we obtain an infinite sequence of finite sequences $\seq{V^n_i}{i \leq K_n}$, each satisfying $(\mathbf A_n)$ and $(\mathbf B_n)$, and also satisfying $(\mathbf C_n)$ when $n > 0$. By concatenating these infinitely many finite paths, we obtain a single infinite sequence: that is, define $\seq{V_m}{m \in \w}$ by setting
$$V_m = V^n_i \ \text{ when } m = K_0+K_1+\dots+K_{n-1}+n+i.$$
Next, define $f: \w \to \w$ by setting
$$f(m) = \min (V_m \setmins \{f(0),f(1),\dots,f(m-1)\})$$
for all $m \in \w$. This function is well-defined because each $V_m$ is infinite. This function is clearly injective, and we claim that it is also surjective. For suppose $f$ is not surjective, and let $\ell$ be the least natural number not in the image of $f$. There is some $n \in \w$ such that every $\ell' < \ell$ is equal to $f(j)$ for some $j < K_0+K_1+\dots+K_{n-1}+n = N$. For every $i \leq K_n$, $V_{N+i} = V^n_i$. As $\P_n$ is a partition of $\w$, there is some $A \in \P_n$ with $m \in A$. Because a transitive graph has no isolated vertices, and because $\mathsf{G}_n$ is transitive by Lemma~\ref{thm:equiv}, condition $(\mathbf B_n)$ implies $\set{i < K_n}{V_i^n = A} \neq \0$. If $i$ is the least index $<\!K_n$ such that $A = V^n_i$, then our definition of $f$ implies 
$$f(N+i) = \min(V^n_i \setmins\{f(0),f(1),\dots,f(N+i-1)\}) = m.$$
This contradicts our choice of $m$ and proves $f$ is surjective. Thus $f$ is a bijection $\w \to \w$.

Let $h$ denote the function $f(n) \mapsto f(n+1)$.
Roughly, we may think of the bijection $f$ as a relabelling of the points of $\w$, and then think of $h$ as the (relabelled) successor map. 
We claim that 
$$\textstyle h^* \,\in\, \Iso(\s) \cap \bigcap_{n \in \w}\gen{\V_n,\mathsf{G}_n} \,=\, \bigcap_{n \in \w}\gen{\V_n,\mathsf{G}_n}_\s \,\sub\, V \cap \bigcap_{n \in \w}U_n.$$

Observe that $\Iso(\s) \cap \bigcap_{n \in \w}\gen{\V_n,\mathsf{G}_n} \,=\, \bigcap_{n \in \w}\gen{\V_n,\mathsf{G}_n}_\s$ by definition, and $\bigcap_{n \in \w}\gen{\V_n,\mathsf{G}_n}_\s \,\sub\, V \cap \bigcap_{n \in \w}U_n$ by our construction of the $\V_n$ and $\mathsf{G}_n$ (specifically by property $(3)$ stated in their construction). So it remains to show that $h^* \in \Iso(\s)$ and that $h^* \in \gen{\V_n,\mathsf{G}_n}$ for every $n \in \w$.

Note that $f$ and $h$ are mod-finite permutations of $\w$, so $f^*,h^* \in \autstar$. To see that $h^* \in \Iso(\s)$, let $s$ denote the successor function $n \mapsto n+1$ and note that $f \circ s = h \circ f$. This implies 
$f^* \circ s^* = h^* \circ f^*$.
But $s^* = \s$, so this shows that $f^*$ is an isomorphism from $\s$ to $h^*$. Hence $h^* \in \Iso(\s)$.

It remains to show $h^* \in \gen{\V_n,\mathsf{G}_n}$ for all $n \in \w$.
Fix $n \in \w$. For any two $A,B \in \P_n$ (not necessarily distinct), let 
$$E_{A \to B} = \set{f(m)}{f(m) \in A \text{ and } f(m+1) \in B}.$$

\begin{claim}
$E_{A \to B}$ is infinite if and only if $A^* \tox{\mathsf{G}_n} B^*$.
\end{claim}
\begin{proof}[Proof of claim]
Suppose $\tilde n \geq n$, and consider the following two statements:
\vspace{1mm}
\begin{minipage}{\textwidth}
\begin{enumerate}[(i)]
\item $A^* \tox{\mathsf{G}_n}B^*$
\item $\tilde A^* \tox{\mathsf{G}_{\tilde n}} \tilde B^*$ for some $\tilde A,\tilde B \in \P_{\tilde n}$ such that $\tilde A \sub A$ and $\tilde B \sub B$.
\end{enumerate}
\end{minipage}
\vspace{1mm}
We claim that $(i)$ and $(ii)$ are equivalent. To see this, recall that $\0 \neq \gen{\V_{\tilde n},\mathsf{G}_{\tilde n}} \sub \gen{\V_n,\mathsf{G}_n}$ and fix some $g \in \gen{\V_{\tilde n},\mathsf{G}_{\tilde n}}$. Because $\mathcal P_{\tilde n}$ refines $\P_n$,
\begin{align*}
A^* \tox{\mathsf{G}_{n}} B^* \quad  \Leftrightarrow& \quad g[A^*] \cap B^* \neq \0 \\
\Leftrightarrow& \quad g[\tilde A^*] \cap \tilde B^* \neq \0 \text{ for some } \tilde A,\tilde B \in \P_{\tilde n} \text{ with } \tilde A \sub A, \tilde B \sub B 
\end{align*}
so that $(i)$ and $(ii)$ are equivalent as claimed.

$E_{A \to B}$ is infinite if and only if there are infinitely many $m \in \w$ such that $f(m) \in A$ and $f(m+1) \in B$. 
Recall that $f$ was defined from the sequence $\seq{V_m}{m \in \w}$ so that $f(m) \in V_m$ for all $m$.
For all but finitely many $m$, $V_m$ is a member of a partition $\P_{\tilde n}$ for some $\tilde n \geq n$, which means that $\P_{\tilde n}$ refines $\P_n$; thus $V_m$ is either contained in or disjoint from each of $A$ and $B$.
Consequently, $E_{A \to B}$ is infinite if and only if there are infinitely many $m$ such that $V_m \sub A$ and $V_{m+1} \sub B$. 

Suppose $A^* \tox{\mathsf{G}_n} B^*$. Let $\tilde n > n$. 
Because $(i)$ and $(ii)$ are equivalent, there are some $\tilde A,\tilde B \in \P_{\tilde n}$ with $\tilde A \sub A$ and $\tilde B \sub B$, such that $\tilde A^* \tox{\mathsf{G}_{\tilde n}} \tilde B^*$. By property $(\mathbf B_{\tilde n})$, there is some $i < K_{\tilde n}$ such that $V^{\tilde n}_i = \tilde A$ and $V^{\tilde n}_{i+1} = \tilde B$; so taking $m = K_0+\dots+K_{\tilde n-1}+\tilde n+i$, we get $V_m = \tilde A \sub A$ and $V_{m+1} = \tilde B \sub B$. Because this holds for every $\tilde n \geq n$, there are infinitely many $m$ such that $V_m \sub A$ and $V_{m+1} \sub B$. By the previous paragraph, this means $E_{A \to B}$ is infinite.

Conversely, suppose $E_{A \to B}$ is infinite; equivalently, suppose there are infinitely many $m$ such that $V_m \sub A$ and $V_{m+1} \sub B$. In particular, there is then some such $m$ with $m > K_0+K_1+\dots+K_{n-1}+n$. We consider two cases, according to whether $V_m$ and $V_{m+1}$ come from the same partition $\P_{\tilde n}$ or not. For the first case, suppose $V_m,V_{m+1} \in \P_{\tilde n}$ for some $\tilde n \geq n$. Then condition $(\mathbf A_{\tilde n})$ implies that $V_m^* \tox{\mathsf{G}_{\tilde n}} V_{m+1}^*$, and (using the assertion $(ii) \Rightarrow (i)$ proved above, with $\tilde A = V_m$ and $\tilde B = V_{m+1}$) this implies $A^* \tox{\mathsf{G}_n}B^*$. In the other case, we must have $V_m \in \P_{\tilde n-1}$ and $V_{m+1} \in \P_{\tilde n}$ for some $\tilde n > n$. But then condition $(\mathbf C_{\tilde n})$ implies that $V_m^* \tox{\mathsf{G}_{\tilde n-1}} \tilde B^*$ for some $\tilde B \in \P_{\tilde n-1}$ with $V_{m+1} \sub \tilde B$. Because $V_{m+1} \sub B$ and $\tilde B$ comes from a partition $\P_{\tilde n-1}$ that refines the partition $\P_n$ containing $B$, we have $V_{m+1} \sub \tilde B$. Hence (using the assertion $(ii) \Rightarrow (i)$ proved above, with $\tilde A = V_m$ and $\tilde B$ as in the previous sentence), this implies $A^* \tox{\mathsf{G}_n}B^*$.
\end{proof}

Returning to the proof of the lemma, observe that by our definition of $h$, if $A,B \in \P_n$ then
$$n \in A \text{ and } h(n) \in B \ \ \Leftrightarrow \ \ n \in E_{A \to B}.$$
which implies $h[A] \cap B$ is infinite if and only if $E_{A \to B}$ is infinite.
From this and the claim above, it follows that for $A,B \in \P_n$,
\begin{align*}
h^*[A^*] \cap B^* = (h[A] \cap B)^* \neq \0 \ &\Leftrightarrow \ h[A] \cap B \text{ is infinite} \\ 
&\Leftrightarrow \ E_{A \to B} \text{ is infinite} \ \Leftrightarrow \ A^* \tox{\mathsf{G_n}} B^*.
\end{align*}
Because $\V_n = \set{A^*}{A \in \P_n}$, this shows that $\mathsf{Hit}(\V_n,h^*) = \mathsf{G}_n$. Therefore $h^* \in \gen{\V_n,\mathsf{G}_n}$. As $n$ was arbitrary, $h^* \in \bigcap_{n \in \w}\gen{\V_n,\mathsf{G}_n}$ as claimed.
\end{proof}

Notice that in the previous lemma, we actually showed a bit more than was stated: every countable intersection of dense open subsets of $\Iso(\s)$ contains a dense set of \emph{trivial} maps. Let us point out that, by modifying the proof of Theorem 2.3 in \cite{BDHHMU}, one may show that every non-empty $G_\delta$ subset of $\autstar$ contains a trivial map. Thus it is not entirely surprising that the statement of Lemma~\ref{lem:baire} can be strengthened in this way.

\begin{lemma}\label{lem:baire2}
$\Iso(\s^{-1})$ is a Baire space. That is, if $\set{U_n}{n \in \w}$ is a collection of dense open subsets of $\Iso(\s^{-1})$, then $\bigcap_{n < \w}U_n$ is dense in $\Iso(\s^{-1})$. 
\end{lemma}
\begin{proof}
Recall that $\H(\w^*)$ is a topological group by Proposition~\ref{prop:group}. In particular, the map $h \mapsto h^{-1}$ is a self-homeomorphism of $\H(\w^*)$, and it restricts to a homeomorphism $\Iso(\s) \to \Iso(\s^{-1})$. Thus this lemma follows from the previous one.
\end{proof}

Recall that a subset $B$ of a topological space $X$ has the \emph{property of Baire} if $B = U \triangle M$ for some open $U \sub X$ and meager $M \sub X$. Recall also that, in any topological space $X$, the sets having the property of Baire form a $\s$-algebra containing all the Borel sets.


If $B,C,D \sub X$, then $B$ \emph{separates} $C$ and $D$ if it contains one and misses the other: i.e., if either $B \supseteq C$ and $B \cap D = \0$, or else $B \supseteq D$ and $B \cap C = \0$.

\begin{theorem}\label{thm:main}
No Borel subset of $\autstar$ separates $\Iso(\s)$ and $\Iso(\s^{-1})$.
\end{theorem}



\begin{proof}
Because the class of Borel sets is closed under taking complements, if a Borel set separates $\Iso(\s)$ and $\Iso(\s^{-1})$, then there is a Borel set containing $\Iso(\s)$ and disjoint from $\Iso(\s^{-1})$.

Suppose $B_0 \sub \autstar$ is Borel and that $B_0 \supseteq \Iso(\s)$. We shall show that $B_0 \cap \Iso(\s^{-1}) \neq \0$. 

Let $K = \closure{\Iso(\s) \cup \Iso(\sinverse)}$. Then $B = B_0 \cap K$ is still a Borel set containing $\Iso(\s)$; moreover, $B$ is relatively Borel in the subspace $K$ of $\autstar$. In particular, $B$ has the property of Baire in $K$. Fix a relatively open $V \sub K$ and a relatively meager $M \sub K$ such that $B = V \triangle M$.

Generally, if $D$ is dense in (some space) $E$, then any somewhere dense subset of $D$ is somewhere dense in $E$. Consequently, if $X$ is nowhere dense in $E$ then $X \cap D$ is nowhere dense in $D$. It follows that if $X$ is meager in $E$ then $X \cap D$ is meager in $D$.

By Theorem~\ref{thm:equiv}, every open subset of $\autstar$ meeting $\Iso(\s^{-1})$ also meets $\Iso(\s)$. It follows that $\Iso(\s)$ is dense in $K$. 
By the previous paragraph, $M \cap \Iso(\s)$ is meager in $\Iso(\s)$. By Lemma~\ref{lem:baire}, this means $M \not\supseteq \Iso(\s)$. Therefore $V \neq \0$.

By the same argument, $\Iso(\s^{-1})$ is dense in $K$, and $M \cap \Iso(\s^{-1})$ is meager in $\Iso(\s^{-1})$. By Lemma~\ref{lem:baire2}, this implies $\Iso(\s^{-1}) \setmins (M \cap \Iso(\s^{-1}))$ is dense in $\Iso(\s^{-1})$.
Hence 
$$(V \setmins M) \cap \Iso(\s^{-1}) = (V \cap \Iso(\s^{-1})) \setmins (M \cap \Iso(\s^{-1})) \neq \0.$$
Hence $B \cap \Iso(\s^{-1}) = (V \triangle M) \cap \Iso(\s^{-1}) \supseteq (V \setmins M) \cap \Iso(\s^{-1}) \neq \0$. As $B_0 \supseteq B$, it follows that $B_0 \cap \Iso(\s^{-1}) \neq \0$. Hence no Borel subset of $\H(\w^*)$ can contain $\Iso(\s)$ without including some of $\Iso(\s^{-1})$.
\end{proof}

Suppose $f: X \to X$ and $g: Y \to Y$ are topological dynamical systems. We say that $g$ is a \emph{quotient} of $f$ if there is a continuous surjection $q: X \to Y$ such that $q \circ f = g \circ q$.

\vspace{.5mm}
\begin{center}
\begin{tikzpicture}[xscale=.85,yscale=.85]

\node at (5.1,0.05) {$\w^*$};
\node at (7.1,0.05) {$\w^*$};
\node at (7.1,2.05) {$\w^*$};
\node at (5.1,2.05) {$\w^*$};
\draw[->] (5.4,2) -- (6.7,2); \node at (6,2.25) {\small $f$};
\draw[->] (5.4,0) -- (6.7,0); \node at (6,-.22) {\small $g$};
\draw[<<-] (7,.35) -- (7,1.68); \node at (7.2,1) {\small $q$};
\draw[<<-] (5,.35) -- (5,1.68); \node at (4.8,1) {\small $q$};
\node at (10,1.25) {\large $f \quotient g$};

\end{tikzpicture}
\end{center}
\vspace{-.5mm}

\noindent This is the standard notion of ``continuous image'' in the category of topological dynamical systems. Recently, in \cite{BrianAOLS}, the quotients of the shift map $\s$ were characterized as follows:

\begin{theorem}\label{thm:brian}
Suppose $X$ is a compact Hausdorff space with weight $\leq\!\aleph_1$. Then $f \in \H(X)$ is a quotient of the shift map $\s \in \autstar$ if and only if $f$ is chain transitive.
\end{theorem}

In \cite{BrianAOLS} it was asserted that this theorem gives a ``simple'' characterization of the quotients of the shift map having weight $\leq\!\aleph_1$. The topological perspective employed here allows us to formalize this assertion:

\begin{corollary}
If $X$ is a compact Hausdorff space with weight $\leq\!\aleph_1$, then 
$$\set{h \in \H(X)}{h \text{ is a quotient of }\s}$$
is a closed subset of $\H(X)$. In particular, if \ch holds then
$$\quot(\s) = \set{h \in \autstar}{h \text{ is a quotient of }\s}$$
is closed in $\autstar$.
\end{corollary}
\begin{proof}
This follows immediately from Theorem~\ref{thm:brian} and Corollary~\ref{cor:shiftisct}.
\end{proof}

It is customary, in the study of $\w^*$, that when one observes some consistent behavior under \ch, one should investigate whether the ``opposite'' behavior occurs under \ocama. 
The topological perspective employed here gives us a framework for doing just this.

\begin{theorem}\label{thm:ocama}
Assuming \ocama, 
$$\quot(\s) = \set{h \in \autstar}{h \text{ is a quotient of }\s}$$
is not Borel in $\autstar$.
\end{theorem}
We recall that \ocama implies $\mathrm{weight}(\w^*) = \continuum = \aleph_2$. Thus this theorem shows, among other things, that the $\aleph_1$ of the previous corollary cannot be improved to an $\aleph_2$ in general. (However, it can be improved to an $\aleph_2$ consistently, e.g. if $\aleph_2 < \pseudo$; see \cite[Theorem 5.10]{BrianAOLS}.)
\begin{proof}[Proof of Theorem~\ref{thm:ocama}]
By Lemma~\ref{lem:ct}, $\s$ and $\s^{-1}$ are chain transitive, and these are (up to isomorphism) the only chain transitive trivial maps in $\autstar$. By a result of Veli\v{c}kovi\'c mentioned already in the introduction, \ocama implies that all members of $\autstar$ are trivial. Thus the chain transitive members of $\autstar$ are precisely those in $\Iso(\s) \cup \Iso(\s^{-1})$.

Every quotient of a chain transitive dynamical system is chain transitive (see, e.g., \cite{Akin}). Thus, by the previous paragraph, if there are no nontrivial members of $\autstar$ then $\quot(\s) \sub \Iso(\s) \cup \Iso(\s^{-1})$. 

Building on work of Farah \cite{Farah}, it was proved in \cite[Theorem 5.7]{BrianAOLS} that \ocama implies $\s^{-1}$ is not a quotient of $\s$. Thus, by the previous paragraph, \ocama implies $\quot(\s) = \Iso(\s)$.

As mentioned in the introduction, van Douwen proved in \cite{vanDouwen} that if there are no nontrivial maps in $\autstar$ then $\Iso(\s) \neq \Iso(\s^{-1})$. In particular, \ocama implies $\Iso(\s) \neq \Iso(\s^{-1})$. By Theorem~\ref{thm:main}, this implies neither $\Iso(\s)$ nor $\Iso(\s^{-1})$ is Borel in $\autstar$. 
\end{proof}

\end{document}